\documentclass[a4paper,11pt,leqno]{amsart}
\usepackage[cp1250]{inputenc}
\usepackage{ifthen}
\usepackage{amssymb}
\usepackage{amsmath}
\usepackage{amsfonts}
\usepackage{latexsym}

\numberwithin{equation}{section}
\newtheorem{theorem}{Theorem}[section]
\theoremstyle{plain}

\newtheorem{lemma}{Lemma}[section]

\numberwithin{equation}{section}
\input{tcilatex}

\begin{document}
\title[The radius of uniform convexity of Bessel functions]{The radius of
uniform convexity of Bessel functions}
\author{Erhan Deniz}
\address{Department of Mathematics, Faculty of Science and Letters, Kafkas
University, Kars, Turkey}
\email{edeniz36@gmail.com}
\author{Róbert Szász}
\address{Department of Mathematics and Informatics, Sapientia Hungarian
University of Transylvania, Târgu-Mureş, Romania}
\email{rszasz@ms.sapientia.ro}
\keywords{Normalized Bessel functions of the fist kind, uniformly convex
function, radius of uniform convexity, zeros of Bessel functions\\
{\indent\textrm{2010 }}\ \textit{Mathematics Subject Classification:}
Primary 33C10, Secondary 30C45.}

\begin{abstract}
In this paper, we determine the radius of uniform convexity for three kinds
of normalized Bessel functions of the first kind. In the mentioned cases the
normalized Bessel functions are uniformly convex on the determined disks.
Moreover, necessary and sufficient conditions are given for the parameters
of the three normalized functions such that they to be uniformly convex in
the open unit disk. The basic tool of this study is the development of
Bessel functions in function series.
\end{abstract}

\maketitle

\section{Introduction}

It is well known that the concepts of convexity, starlikeness,
close-to-convexity and uniform convexity including necessary and sufficient
conditions, have a long history as a part of geometric function theory. In
1945, Pólya and Szegö \cite{Ps} found the necessary and sufficient
conditions of convexity and starlikeness for analytic functions which were
further generalized by Royster \cite{Roy} and Bernardi \cite{Be}. After that
several authors contributed to this literature by doing generalizations of
the previously developed criterions and also by introducing some new ones,
for details, see \cite{God,J,Oz}. Significant contributions to the same were
made by Mocanu \cite{Mo1,Mo2}, Obradović \cite{Ob} and Owa et.al \cite{Ow}
by introducing more applicable criterions for close-to-convexity, convexity
and starlikeness. Tuneski \cite{Tu} used the method of differential
subordination to find the conditions for starlikeness of analytic functions.
In 1993, R\o nning \cite{10} determined necessary and sufficient conditions
of analytic functions to be uniformly convex in the open unit disk, while in
2002 Ravichandran \cite{Rav} also presented more simple criterions for
uniform convexity. Silverman \cite{Sl} investigated the properties of
functions defined in terms of the quotient of analytic representations of
convex and starlike functions which was then improved by Obradović and
Tuneski \cite{Ob1} and Tuneski \cite{Tu1}. Recent works on certain
criterions of convexity and starlikeness can be found in \cite%
{Nu,Nu1,Nu2,Tu2}.

On the other hand, one of the most important applications of the concepts of
convexity, starlikeness, close-to-convexity and uniform convexity is to find
the necessary and sufficient condition of theirs for hypergeometric and
Bessel functions. In 1961, Merkes and Scott \cite{Mer} investigated the
starlikeness and univalence of Gaussian hypergeometric functions by using
continued-fraction representations, while in 1986 Ruscheweyh and Singh \cite%
{Ru} obtained the exact order of starlikeness by using the same technique.
Owa and Srivastava \cite{OwS} investigated the geometric properties of
generalized hypergeometric functions using the well known Jack's lemma.
Miller and Mocanu \cite{Mi} employed the method of differential
subordinations to investigate the local univalence, starlikeness and
convexity of certain hypergeometric functions. Silverman \cite{Sil} in 1993
also investigated the starlikeness and convexity of Gaussian hypergeometric
functions, while in 1998 and 2001 Ponnusamy and Vuorinen \cite{Po1,Po2}
presented some generalizations of the results of Miller and Mocanu, and
determined conditions of close-to-convexity of confluent (or Kummer) and
Gaussian hypergeometric functions, respectively. Küstner \cite{Ku} by using
among others the continued fraction of C.F. Gauss determined the order of
convexity and starlikeness of hypergeometric functions. Many authors have
determined the necessary and sufficient conditions for the hypergeometric
functions to be uniform convexity \cite{Kim,Sri,Sw}. An extensive
bibliography and history of convexity, starlikeness and close-to-convexity
for Bessel functions can be found in the two sections of Chap. 3.

Let $U(z_{0},r)=\{z\in \mathbb{C}:|z-z_{0}|<r\}$ denote the disk of radius $%
r $ and center $z_{0}.$ We denote by $U(r)=U(0,r)$ and by $U=U(0,1)=\{z\in 
\mathbb{C}:|z|<1\}.$ Let $(a_{n})_{n\geq 2}$ be a sequence of complex
numbers with 
\begin{equation*}
d=\limsup_{n\rightarrow \infty }|a_{n}|^{\frac{1}{n}}\geq 0,\ \text{and}\
r_{f}=\frac{1}{d}.
\end{equation*}%
If $d=0$ then $r_{f}=+\infty .$ The power series 
\begin{equation}
f(z)=z+\sum_{n=2}^{\infty }{a}_{n}z^{n}  \label{1}
\end{equation}%
defines an analytic function $f:U(r_{f})\rightarrow \mathbb{C}.$\newline
For $r\in (0,r_{f})$ we say that the function $f$ is starlike in the disk $%
U(r)=\{z\in {%
\mathbb{C}
}:|z|<r\}$ if $f$ is univalent in $U(r),$ and $f(U(r))$ is a starlike domain
with respect to $0$ in $\mathbb{C}$. Regarding the starlikeness of the
function $f$ the following equivalency holds 
\begin{equation*}
f\ \text{is starlike in}\ U(r)\ \text{if\ and only\ if}\ \func{Re}\left( 
\frac{zf^{\prime }(z)}{f(z)}\right) >0,\ \ z\in {U(r)}.
\end{equation*}%
We define by 
\begin{equation*}
r_{f}^{\ast }=\sup \left\{ r\in (0,r_{f}):\ \func{Re}\left( \frac{zf^{\prime
}(z)}{f(z)}\right) >0,\ \ z\in {U}(r)\right\}
\end{equation*}%
the radius of starlikeness of the function $f.$ \newline
The radius of convexity is defined in a similar manner. We say that a
function $f$ of the form (\ref{1}) is convex if $f$ is univalent and $%
f(U(r)) $ is a convex domain in $\mathbb{C}$. An analytic description of
this definition is 
\begin{equation*}
f\in \mathcal{A}\ \text{is convex\ if and only if}\ \func{Re}\left( 1+\frac{%
zf^{\prime \prime }(z)}{f^{\prime }(z)}\right) >0,\ \ z\in {U(r)}.
\end{equation*}%
The radius of convexity of the function $f$ is defined by 
\begin{equation*}
r_{f}^{c}=\sup \left\{ r\in (0,r_{f}):\ \func{Re}\left( 1+\frac{zf^{\prime
\prime }(z)}{f^{\prime }(z)}\right) >0,\ \ z\in {U}(r)\right\} .
\end{equation*}%
We will give a few definitions and results in the next section which we will
use further on to determine the radius of uniform convexity.

\section{Preliminaries}

In the following we deal with the class of the uniformly convex functions.
Goodman in \cite{Goo} introduced the concept of uniform convexity for
functions of the form (\ref{1}). A function $f$ is said to be uniformly
convex in $U(r)$ if $f$ is of the form (\ref{1}), it is convex, and has the
property that for every circular arc $\gamma $ contained in $U(r)$, with
center $\varsigma $, also in $U(r),$ the arc $f(\gamma )$ is convex. An
analytic description of the uniformly convex functions is given in the next
theorem, which is a slight modification of Theorem 1 from \cite{10}.

\begin{theorem}
\label{T1} Let $f$ be a function of the form $f(z)=z+\sum_{n=2}^{\infty
}a_{n}z^{n},$ and analytic in the disk $U(r).$ The function $f$ is uniformly
convex in the disk $U(r)$ if and only if 
\begin{equation}
\func{Re}\left( 1+\frac{zf^{\prime \prime }(z)}{f^{\prime }(z)}\right)
>\left\vert \frac{zf^{\prime \prime }(z)}{f^{\prime }(z)}\right\vert ,\ \ \
z\in {U(r).}  \label{d1}
\end{equation}
\end{theorem}

This theorem makes possible to determine the radius of uniform convexity of
Bessel functions. The radius of uniform convexity is defined by 
\begin{equation*}
r_{f}^{uc}=\sup \left\{ r\in (0,r_{f}):\ \func{Re}\left( 1+\frac{zf^{\prime
\prime }(z)}{f^{\prime }(z)}\right) >\left\vert \frac{zf^{\prime \prime }(z)%
}{f^{\prime }(z)}\right\vert ,\ \ z\in {U}(r)\right\} .
\end{equation*}%
In order to prove the main results later on, we need the following lemma.

\begin{lemma}
\label{L1} i. If $a>b>r\geq |z|,$ and $\lambda \in \lbrack 0,1],$ then 
\begin{equation}
\left\vert \frac{z}{b-z}-\lambda \frac{z}{a-z}\right\vert \leq \frac{r}{b-r}%
-\lambda \frac{r}{a-r}.  \label{2}
\end{equation}%
Very simple consequences of this inequality are the followings 
\begin{equation}
\func{Re}\left( \frac{z}{b-z}-\lambda \frac{z}{a-z}\right) \leq \frac{r}{b-r}%
-\lambda \frac{r}{a-r}  \label{3}
\end{equation}%
and 
\begin{equation}
\func{Re}\left( \frac{z}{b-z}\right) \leq \left\vert \frac{z}{b-z}%
\right\vert \leq \frac{r}{b-r}.  \label{4}
\end{equation}%
ii. If $b>a>r\geq |z|,$ then%
\begin{equation}
\left\vert \frac{1}{(a+z)(b-z)}\right\vert \leq \frac{1}{(a-r)(b+r)}.
\label{41}
\end{equation}
\end{lemma}

\begin{proof}
\textit{i. }According to the maximum principle for harmonic functions we
have to prove inequality (\ref{2}) only in case $z=re^{i\theta }.$ In this
case the inequality is equivalent to 
\begin{equation}
\left\vert \frac{1}{b-re^{i\theta }}-\lambda \frac{1}{a-re^{i\theta }}%
\right\vert \leq \frac{1}{b-r}-\lambda \frac{1}{a-r}.  \label{5}
\end{equation}%
Denoting $a_{1}=\frac{a}{r},b_{1}=\frac{b}{r}$ inequality (\ref{5}) becomes 
\begin{equation}
\left\vert \frac{1}{b_{1}-e^{i\theta }}-\lambda \frac{1}{a_{1}-e^{i\theta }}%
\right\vert \leq \frac{1}{b_{1}-1}-\lambda \frac{1}{a_{1}-1},\
a_{1}>b_{1}>1,\ \theta \in \lbrack 0,2\pi ].  \label{6}
\end{equation}%
Let the function $\varphi :[0,1]\rightarrow \mathbb{R}$ be defined by 
\begin{equation*}
\varphi (\lambda )=\left( \frac{1}{b_{1}-1}-\lambda \frac{1}{a_{1}-1}\right)
^{2}-\left\vert \frac{1}{b_{1}-e^{i\theta }}-\lambda \frac{1}{%
a_{1}-e^{i\theta }}\right\vert ^{2}.
\end{equation*}%
By denoting $t=\cos \theta ,$ the function $\varphi $ can be rewritten in
the following form%
\begin{eqnarray*}
\varphi (\lambda ) &=&\left( \frac{1}{b_{1}-1}-\lambda \frac{1}{a_{1}-1}%
\right) ^{2}-\frac{1}{b_{1}^{2}-2b_{1}t+1}-\lambda ^{2}\frac{1}{%
a_{1}^{2}-2a_{1}t+1} \\
&&+2\lambda \frac{a_{1}b_{1}-(a_{1}+b_{1})t+1}{%
(a_{1}^{2}-2a_{1}t+1)(b_{1}^{2}-2b_{1}t+1)}.
\end{eqnarray*}%
We have that 
\begin{eqnarray*}
\varphi ^{\prime }(\lambda ) &=&2\left( \frac{1}{b_{1}-1}-\lambda \frac{1}{%
a_{1}-1}\right) \left( -\frac{1}{a_{1}-1}\right) -2\lambda \frac{1}{%
a_{1}^{2}-2a_{1}t+1} \\
&&+2\frac{a_{1}b_{1}-(a_{1}+b_{1})t+1}{%
(a_{1}^{2}-2a_{1}t+1)(b_{1}^{2}-2b_{1}t+1)}
\end{eqnarray*}%
and 
\begin{equation*}
\varphi ^{\prime \prime }(\lambda )=2\frac{1}{(a_{1}-1)^{2}}-2\frac{1}{%
a_{1}^{2}-2a_{1}t+1}>0.
\end{equation*}%
Thus $\varphi ^{\prime }$ is strictly increasing on $[0,1],$ and
consequently if $\varphi ^{\prime }(1)<0,$ then $\varphi ^{\prime }(\lambda
)<0,\ \lambda \in \lbrack 0,1].$\newline
Some calculations lead to 
\begin{equation*}
\varphi ^{\prime }(1)=2(a_{1}-b_{1})\left( \frac{b_{1}-t}{%
(a_{1}^{2}-2a_{1}t+1)(b_{1}^{2}-2b_{1}t+1)}-\frac{1}{(a_{1}-1)^{2}(b_{1}-1)}%
\right) .
\end{equation*}%
If 
\begin{equation*}
m(t)=\frac{b_{1}-t}{(a_{1}^{2}-2a_{1}t+1)(b_{1}^{2}-2b_{1}t+1)},
\end{equation*}%
then 
\begin{equation*}
\varphi ^{\prime }(1)=2(a_{1}-b_{1})\left( m(t)-m(1)\right) .
\end{equation*}%
Since 
\begin{equation*}
m^{\prime }(t)=\frac{%
(a_{1}^{2}-2a_{1}t+1)(b_{1}^{2}-1)+2a_{1}(b_{1}^{2}-2b_{1}t+1)(b_{1}-t)}{%
(a_{1}^{2}-2a_{1}t+1)^{2}(b_{1}^{2}-2b_{1}t+1)^{2}}>0,
\end{equation*}%
it follows that $m(t)\leq {m}(1)$ and consequently $\varphi ^{\prime }(1)<0,$
and $\varphi ^{\prime }(\lambda )<0,\ \lambda \in \lbrack 0,1].$ This
implies that $\varphi $ is strictly decreasing and $\varphi (\lambda )\geq
\varphi (1)$ or equivalently%
\begin{eqnarray*}
&&\left( \frac{1}{b_{1}-1}-\lambda \frac{1}{a_{1}-1}\right) ^{2}-\left( 
\frac{1}{b_{1}-e^{i\theta }}-\lambda \frac{1}{a_{1}-e^{i\theta }}\right) ^{2}
\\
&\geq &\left( \frac{1}{b_{1}-1}-\frac{1}{a_{1}-1}\right) ^{2}-\left( \frac{1%
}{b_{1}-e^{i\theta }}-\frac{1}{a_{1}-e^{i\theta }}\right) ^{2},\ \ \ \theta
\in \lbrack 0,2\pi ].
\end{eqnarray*}%
Thus in order to prove (\ref{6}) we have to show that 
\begin{equation}
\frac{1}{b_{1}-1}-\frac{1}{a_{1}-1}\geq \left\vert \frac{1}{b_{1}-e^{i\theta
}}-\frac{1}{a_{1}-e^{i\theta }}\right\vert ,\ a_{1}>b_{1}>1,\ \theta \in
\lbrack 0,2\pi ]  \label{7}
\end{equation}%
or equivalently%
\begin{equation}
\frac{1}{(a_{1}-1)(b_{1}-1)}\geq \frac{1}{|(a_{1}-e^{i\theta
})(b_{1}-e^{i\theta })|},\ a_{1}>b_{1}>1,\ \theta \in \lbrack 0,2\pi ].
\label{71}
\end{equation}

Since the inequalities $\left\vert {a_{1}}-e^{i\theta }\right\vert \geq {%
a_{1}-1}$ and $|{b_{1}}-e^{i\theta }|\geq {b_{1}-1}$ for $a_{1}>b_{1}>1,\
\theta \in \lbrack 0,2\pi ]$ holds we get (\ref{71}) and thus (\ref{7}) . We
mention that (\ref{3}) and (\ref{4}) have been proved in \cite{Ba1} using a
direct method.

\textit{ii. }According to the maximum principle it is enough to prove the
inequality (\ref{41}) in case of $z=re^{i\theta },$ that is 
\begin{equation}
\left\vert \frac{1}{(a+re^{i\theta })(b-re^{i\theta })}\right\vert \leq 
\frac{1}{(a-r)(b+r)}.  \label{eq6}
\end{equation}%
Denoting $\alpha =\frac{a}{r}$ and $\beta =\frac{b}{r},$ the inequality (\ref%
{eq6}) can be rewritten as follows 
\begin{equation}
\left\vert \frac{1}{(\alpha +e^{i\theta })(\beta -e^{i\theta })}\right\vert
\leq \frac{1}{(\alpha -1)(\beta +1)},  \label{eq7}
\end{equation}%
where $\beta >\alpha >1.$ We will prove the inequality (\ref{eq7}). If $%
t=\cos \theta ,$ then this inequality will be equivalent to 
\begin{equation}
(\beta ^{2}-2\beta {t}+1)(\alpha ^{2}+2\alpha {t}+1)\geq {(\beta
+1)^{2}(\alpha -1)^{2}},\ \ \beta >\alpha >1,\ t\in \lbrack -1,1].
\label{eq9}
\end{equation}%
In order to prove inequality (\ref{eq9}) we define the function 
\begin{equation*}
u:[-1,1]\rightarrow \mathbb{R},\ \ u(t)=(\beta ^{2}-2\beta {t}+1)(\alpha
^{2}+2\alpha {t}+1).
\end{equation*}%
Since $u^{\prime \prime }(t)=-8\alpha \beta <0,\ t\in \lbrack -1,1]$ it
follows that $u$ is a concave mapping, and consequently 
\begin{equation*}
u(t)\geq \min \{u(1),u(-1)\}=u(-1)=(\beta +1)^{2}(\alpha -1)^{2}.
\end{equation*}%
Thus the proof of the inequality (\ref{41}) is done.\newline
\end{proof}

\section{Main Results}

The Bessel function of the first kind of order $\nu $ is defined by \cite[p.
217]{Olver} 
\begin{equation*}
J_{\nu }(z)=\sum_{n\geq 0}\frac{(-1)^{n}}{n!\Gamma (n+\nu +1)}\left( \frac{z%
}{2}\right) ^{2n+\nu }.
\end{equation*}%
In this paper we deal with the following normalized forms 
\begin{equation*}
f_{\nu }(z)=\left( 2^{\nu }\Gamma (\nu +1)J_{\nu }(z)\right) ^{\frac{1}{\nu }%
}=z-\frac{1}{4\nu (\nu +1)}z^{3}+\dots ,\ \nu \neq 0,
\end{equation*}%
\begin{equation*}
g_{\nu }(z)=2^{\nu }\Gamma (\nu +1)z^{1-{\nu }}J_{\nu }(z)=z-\frac{1}{4(\nu
+1)}z^{3}+\frac{1}{32(\nu +1)(\nu +2)}z^{5}-\dots ,
\end{equation*}%
\begin{equation*}
h_{\nu }(z)=2^{\nu }\Gamma (\nu +1)z^{1-\frac{\nu }{2}}J_{\nu }(\sqrt{z})=z-%
\frac{1}{4(\nu +1)}z^{2}+\dots ,
\end{equation*}%
where $\nu >-1.$ Observe that $f_{\nu },g_{\nu },h_{\nu }\in \mathcal{A}.$
We note that 
\begin{equation*}
f_{\nu }(z)=\exp \left( \frac{1}{\nu }\log \left( 2^{\nu }\Gamma (\nu
+1)J_{\nu }(z)\right) \right) ,
\end{equation*}%
where $\log $ represents the principal branch of the logarithm.

Using the proved inequalities we will determine the smallest value $\delta $
such that the inequality $\nu \geq \delta $ implies the uniform convexity in 
$U$ of different normalized Bessel functions. The radius of uniform
convexity of the Bessel functions of the first kind are also determined.
These two questions are closely connected.

Here and in the sequel $I_{\nu }$ denotes the modified Bessel function of
the fist kind and order $\nu .$ Note that $I_{\nu }(z)=i^{-\nu }J_{\nu }(iz)$
and $I_{\nu }(\sqrt{z})=(-1)^{-\frac{\nu }{2}}J_{\nu }(\sqrt{-z})$.

\subsection{\textbf{The radius of uniform convexity of normalized Bessel
functions}}

As far as we know the first results regarding the starlikeness of Bessel
functions have been given in \cite{Br} and \cite{KrT}. These two papers
initiated a research to study the univalence of Bessel functions and
determine the radius of starlikeness for different kind of normalizations.
These results raise questions about other geometric properties of Bessel
functions, like convexity, uniform convexity and etc. Recently, Baricz et
al. \cite{Ba0}, Baricz and Szász \cite{Ba1} and Baricz et al. \cite{Ba2}
obtained, respectively, the radius of starlikeness of order $\beta ,$ the
radius of convexity of order $\beta $ and the radius of $\alpha -$convexity
of order $\beta $ for the functions $f_{\nu }(z),$ $g_{\nu }(z)$ and $h_{\nu
}(z)$ in the case when $\nu >-1.$ On the other hand, we know that if $\nu
\in (-2,-1),$ then the Bessel function has exactly two purely imaginary
conjugate complex zeros, and all the other zeros are real (see \cite{Wat} p.
483). Thus in order to solve the above radius problems in case $\nu \in
(-2,-1),$ the method which has been used in \cite{Ba0,Ba1,Ba2} is not
applicable directly. In \cite{Sza}, Szász investigated the radius of
starlikeness of order $\beta $ for the functions $g_{\nu }(z)$ and $h_{\nu
}(z)$ in the case when $\nu \in (-2,-1)$ by using some inequalities. Baricz
and Szász \cite{Ba3} obtained the radius of convexity of order $\beta $ for
the functions $g_{\nu }(z)$ and $h_{\nu }(z)$ in the case when $\nu \in
(-2,-1).$ Very Recently, Deniz et al. \cite{De} investigated the radius of $%
\alpha -$convexity of order $\beta $ for the functions $g_{\nu }(z)$ and $%
h_{\nu }(z)$ in the case when $(-2,-1).$ In the paper \cite{1} the radius of
starlikeness and the radius of convexity of $q-$Bessel functions have also
been determined . In this section, we deal with the radius of uniform
convexity for the normalized Bessel functions $f_{\nu }(z),$ $g_{\nu }(z)$
and $h_{\nu }(z)$ in the case when $\nu >-2$ $\left( \nu \neq -1\right) .$

\begin{theorem}
\label{th2} If $\nu >0,$ then the radius of uniform convexity of the
function $f_{\nu }$ is the smallest positive root of the equation 
\begin{equation*}
1+2\frac{(r^{2}-\nu ^{2})J_{\nu }(r)}{rJ_{\nu }^{\prime }(r)}+2\left( 1-%
\frac{1}{\nu }\right) \frac{rJ_{\nu }^{\prime }(r)}{J_{\nu }(r)}=0.
\end{equation*}%
Moreover $r^{uc}(f_{\nu })<r^{c}(f_{\nu })<j_{\nu ,1}^{\prime }<j_{\nu ,1},$
where $j_{\nu ,1}$ and $j_{\nu ,1}^{\prime }$ denote the first positive
zeros of $J_{\nu }$ and $J_{\nu }^{\prime },$ respectively and $r^{c}(f_{\nu
})$ is the radius of convexity of the function $f_{\nu }.$
\end{theorem}

\begin{proof}
Let $j_{\nu ,n}$ and $j_{\nu ,n}^{\prime }$ are the $n$-th positive roots of 
$J_{\nu }$ and $J_{\nu }^{\prime },$ respectively. In \cite[p.11]{Ba1} the
following equality was proved 
\begin{equation*}
1+\frac{zf_{\nu }^{\prime \prime }(z)}{f_{\nu }^{\prime }(z)}=1-\left( \frac{%
1}{\nu }-1\right) \sum_{n\geq 1}\frac{2z^{2}}{j_{\nu ,n}^{2}-z^{2}}%
-\sum_{n\geq 1}\frac{2z^{2}}{j_{\nu ,n}^{\prime 2}-z^{2}}.
\end{equation*}%
We will prove the theorem in two steps.

First suppose $\nu \geq 1.$ In this case we will use the property of the
zeros $j_{\nu ,n}$ and $j_{\nu ,n}^{\prime },$ that interlace according to
the inequalities 
\begin{equation*}
\nu \leq j_{\nu ,1}^{\prime }<j_{\nu ,1}<j_{\nu ,2}^{\prime }<j_{\nu
,2}<j_{\nu ,3}^{\prime }<{\dots }.
\end{equation*}%
Putting $\lambda =1-\frac{1}{\nu },$ inequality (\ref{3}) implies $\func{Re}%
\left( \frac{2z^{2}}{j_{\nu ,n}^{\prime 2}-z^{2}}-\left( 1-\frac{1}{\nu }%
\right) \frac{2z^{2}}{j_{\nu ,n}^{2}-z^{2}}\right) \leq \frac{2r^{2}}{j_{\nu
,n}^{\prime 2}-r^{2}}-\left( 1-\frac{1}{\nu }\right) \frac{2r^{2}}{j_{\nu
,n}^{2}-r^{2}},$ for $|z|\leq {r}<j_{\nu ,1}^{\prime }<j_{\nu ,1},$ and we
get%
\begin{eqnarray}
\func{Re}\left( 1+\frac{zf_{\nu }^{\prime \prime }(z)}{f_{\nu }^{\prime }(z)}%
\right) &=&1-\sum_{n\geq 1}\func{Re}\left( \frac{2z^{2}}{j_{\nu ,n}^{\prime
2}-z^{2}}-\left( 1-\frac{1}{\nu }\right) \frac{2z^{2}}{j_{\nu ,n}^{2}-z^{2}}%
\right)  \label{8} \\
&\geq &1-\sum_{n\geq 1}\left( \frac{2r^{2}}{j_{\nu ,n}^{\prime 2}-r^{2}}%
-\left( 1-\frac{1}{\nu }\right) \frac{2r^{2}}{j_{\nu ,n}^{2}-r^{2}}\right) 
\notag \\
&=&1+\frac{rf_{\nu }^{\prime \prime }(r)}{f_{\nu }^{\prime }(r)}.  \notag
\end{eqnarray}%
On the other hand if in the inequality (\ref{2}) we replace $z$ by $z^{2}$
and we put again $\lambda =1-\frac{1}{\nu },$ then it follows that $%
\left\vert \frac{2z^{2}}{j_{\nu ,n}^{\prime 2}-z^{2}}-\left( 1-\frac{1}{\nu }%
\right) \frac{2z^{2}}{j_{\nu ,n}^{2}-z^{2}}\right\vert \leq \frac{2r^{2}}{%
j_{\nu ,n}^{\prime 2}-r^{2}}-\left( 1-\frac{1}{\nu }\right) \frac{2r^{2}}{%
j_{\nu ,n}^{2}-r^{2}}$ provided that $|z|\leq {r}<j_{\nu ,1}^{\prime
}<j_{\nu ,1}.$ Thus, we have%
\begin{eqnarray}
\left\vert \frac{zf_{\nu }^{\prime \prime }(z)}{f_{\nu }^{\prime }(z)}%
\right\vert &=&\left\vert \sum_{n\geq 1}\left( \frac{2z^{2}}{j_{\nu
,n}^{\prime 2}-z^{2}}-\left( 1-\frac{1}{\nu }\right) \frac{2z^{2}}{j_{\nu
,n}^{2}-z^{2}}\right) \right\vert  \label{9} \\
&\leq &\sum_{n\geq 1}\left\vert \frac{2z^{2}}{j_{\nu ,n}^{\prime 2}-z^{2}}%
-\left( 1-\frac{1}{\nu }\right) \frac{2z^{2}}{j_{\nu ,n}^{2}-z^{2}}%
\right\vert  \notag \\
&\leq &\sum_{n\geq 1}\left( \frac{2r^{2}}{j_{\nu ,n}^{\prime 2}-r^{2}}%
-\left( 1-\frac{1}{\nu }\right) \frac{2r^{2}}{j_{\nu ,n}^{2}-r^{2}}\right) =-%
\frac{rf_{\nu }^{\prime \prime }(r)}{f_{\nu }^{\prime }(r)}.  \notag
\end{eqnarray}%
In the second step we will prove that inequalities (\ref{8}) and (\ref{9})
hold in the case $\nu \in (0,1)$ too. Indeed in the case $\nu \in (0,1)$ the
roots $0<j_{\nu ,n}^{\prime }<j_{\nu ,n}$ are real for every natural number $%
n.$ Inequality (\ref{4}) implies 
\begin{equation*}
\func{Re}\frac{2z^{2}}{j_{\nu ,n}^{\prime 2}-z^{2}}\leq \left\vert \frac{%
2z^{2}}{j_{\nu ,n}^{\prime 2}-z^{2}}\right\vert \leq \frac{2r^{2}}{j_{\nu
,n}^{\prime 2}-r^{2}},\ |z|\leq {r}<j_{\nu ,1}^{\prime }<j_{\nu ,1}
\end{equation*}%
and 
\begin{equation*}
\func{Re}\frac{2z^{2}}{j_{\nu ,n}^{2}-z^{2}}\leq \left\vert \frac{2z^{2}}{%
j_{\nu ,n}^{2}-z^{2}}\right\vert \leq \frac{2r^{2}}{j_{\nu ,n}^{2}-r^{2}},\
|z|\leq {r}<j_{\nu ,1}^{\prime }<j_{\nu ,1}.
\end{equation*}%
Since $\frac{1}{\nu }-1>0,$ the previous inequalities imply that%
\begin{eqnarray}
\func{Re}\left( 1+\frac{zf_{\nu }^{\prime \prime }(z)}{f_{\nu }^{\prime }(z)}%
\right) &=&1-\sum_{n\geq 1}\func{Re}\left( \frac{2z^{2}}{j_{\nu ,n}^{\prime
2}-z^{2}}\right) -\left( \frac{1}{\nu }-1\right) \sum_{n\geq 1}\func{Re}%
\left( \frac{2z^{2}}{j_{\nu ,n}^{2}-z^{2}}\right)  \label{10} \\
&\geq &1-\sum_{n\geq 1}\frac{2r^{2}}{j_{\nu ,n}^{\prime 2}-r^{2}}-\left( 
\frac{1}{\nu }-1\right) \sum_{n\geq 1}\frac{2r^{2}}{j_{\nu ,n}^{2}-r^{2}} 
\notag \\
&=&1+\frac{rf_{\nu }^{\prime \prime }(r)}{f_{\nu }^{\prime }(r)}.\ \ \  
\notag
\end{eqnarray}%
Now if in the second part of inequality (\ref{4} ) we replace $z$ by $z^{2},$
and $b$ by $j_{\nu ,n}^{\prime }$ and by $j_{\nu ,n}$, respectively, then it
follows that $\left\vert \frac{2z^{2}}{j_{\nu ,n}^{\prime 2}-z^{2}}%
\right\vert \leq \frac{2r^{2}}{j_{\nu ,n}^{\prime 2}-r^{2}},$ and $%
\left\vert \frac{2z^{2}}{j_{\nu ,n}^{2}-z^{2}}\right\vert \leq \frac{2r^{2}}{%
j_{\nu ,n}^{2}-r^{2}},$ provided that $|z|\leq {r}<j_{\nu ,1}^{\prime
}<j_{\nu ,1}.$\newline
These two inequalities and the condition $\frac{1}{\nu }-1>0,$ imply%
\begin{eqnarray}
\left\vert \frac{zf_{\nu }^{\prime \prime }(z)}{f_{\nu }^{\prime }(z)}%
\right\vert &=&\left\vert \sum_{n\geq 1}\left( \frac{2z^{2}}{j_{\nu
,n}^{\prime 2}-z^{2}}+\left( \frac{1}{\nu }-1\right) \frac{2z^{2}}{j_{\nu
,n}^{2}-z^{2}}\right) \right\vert  \label{11} \\
&\leq &\sum_{n\geq 1}\left\vert \frac{2z^{2}}{j_{\nu ,n}^{\prime 2}-z^{2}}%
\right\vert +\left( \frac{1}{\nu }-1\right) \sum_{n\geq 1}\left\vert \frac{%
2z^{2}}{j_{\nu ,n}^{2}-z^{2}}\right\vert  \notag \\
&\leq &\sum_{n\geq 1}\left( \frac{2r^{2}}{j_{\nu ,n}^{\prime 2}-r^{2}}%
+\left( \frac{1}{\nu }-1\right) \frac{2r^{2}}{j_{\nu ,n}^{2}-r^{2}}\right) =-%
\frac{rf_{\nu }^{\prime \prime }(r)}{f_{\nu }^{\prime }(r)}.  \notag
\end{eqnarray}%
Finally from (\ref{8}) and (\ref{9}) we infer 
\begin{equation}
\func{Re}\left( 1+\frac{zf_{\nu }^{\prime \prime }(z)}{f_{\nu }^{\prime }(z)}%
\right) -\left\vert \frac{zf_{\nu }^{\prime \prime }(z)}{f_{\nu }^{\prime
}(z)}\right\vert \geq 1+2\frac{rf_{\nu }^{\prime \prime }(r)}{f_{\nu
}^{\prime }(r)},\ |z|\leq {r}<j_{\nu ,1}^{\prime },  \label{111}
\end{equation}%
and (\ref{10}), (\ref{11}) also lead to the previous inequality. The
equality holds if and only if $z=r.$ Thus it follows 
\begin{equation*}
\inf_{|z|<r}\left[ \func{Re}\left( 1+\frac{zf_{\nu }^{\prime \prime }(z)}{%
f_{\nu }^{\prime }(z)}\right) -\left\vert \frac{zf_{\nu }^{\prime \prime }(z)%
}{f_{\nu }^{\prime }(z)}\right\vert \right] =1+2\frac{rf_{\nu }^{\prime
\prime }(r)}{f_{\nu }^{\prime }(r)},\ \ r\in (0,j_{\nu ,1}^{\prime }).
\end{equation*}%
The mapping $\psi _{\nu }:(0,j_{\nu ,1}^{\prime })\rightarrow \mathbb{R}$
defined by%
\begin{equation*}
\psi _{\nu }(r)=1+2\frac{rf_{\nu }^{\prime \prime }(r)}{f_{\nu }^{\prime }(r)%
}=1-2\sum_{n\geq 1}\left( \frac{2r^{2}}{j_{\nu ,n}^{\prime 2}-r^{2}}-\left(
1-\frac{1}{\nu }\right) \frac{2r^{2}}{j_{\nu ,n}^{2}-r^{2}}\right)
\end{equation*}%
is strictly decreasing, $\lim_{r\searrow 0}\psi _{\nu }(r)=1$ and $%
\lim_{r\nearrow {j_{\nu ,n}^{\prime }}}\psi _{\nu }(r)=-\infty .$ Thus it
follows that the equation $1+2\frac{rf_{\nu }^{\prime \prime }(r)}{f_{\nu
}^{\prime }(r)}=0$ has a unique root $r_{0}\in (0,j_{\nu ,n}^{\prime })$ and 
$r_{0}=r^{uc}(f_{\nu }).$ Since $1+2\frac{rf_{\nu }^{\prime \prime }(r)}{%
f_{\nu }^{\prime }(r)}=1+2\frac{rJ_{\nu }^{\prime \prime }(r)}{J_{\nu
}^{\prime }(r)}+2\left( \frac{1}{\nu }-1\right) \frac{rJ_{\nu }^{\prime }(r)%
}{J_{\nu }(r)}$ and using the Bessel differential equation $z^{2}J_{\nu
}^{\prime \prime }(z)+zJ_{\nu }{}^{\prime }(z)+\left( 1-\nu ^{2}\right)
J_{\nu }(z)=0,$ the proof is done.
\end{proof}

\FRAME{dhFU}{3.1168in}{2.002in}{0pt}{\Qcb{The graph of the function $\protect%
\nu \mapsto 1+2\frac{(r^{2}-\protect\nu ^{2})J_{\protect\nu }(r)}{rJ_{%
\protect\nu }^{\prime }(r)}+2\left( 1-\frac{1}{\protect\nu }\right) \frac{%
rJ_{\protect\nu }^{\prime }(r)}{J_{\protect\nu }(r)}$ for $\protect\nu \in
\{0.5,1,1.5,2.5\}$ on $\left[ 0,1\right] $}}{}{Figure 1}{\special{language
"Scientific Word";type "GRAPHIC";display "USEDEF";valid_file "T";width
3.1168in;height 2.002in;depth 0pt;original-width 5.0004in;original-height
3.1142in;cropleft "0";croptop "1";cropright "1";cropbottom "0";tempfilename
'OH5TDI00.wmf';tempfile-properties "PR";}}

\begin{theorem}
\label{th3} \textbf{i.} If $\nu >-1,$ then the radius of uniform convexity
of the function $g_{\nu }$ is the smallest positive root of the equation 
\begin{equation*}
1+2r\frac{(2\nu -1)J_{\nu +1}(r)-rJ_{\nu }(r)}{J_{\nu }(r)-rJ_{\nu +1}(r)}=0.
\end{equation*}%
Moreover, $r_{\alpha }^{c}(g_{\nu })<\alpha _{\nu ,1}<j_{\nu ,1},$ where $%
\alpha _{\nu ,1}$ is the first positive zero of the Dini function $z\mapsto
(1-\nu )J_{\nu }(z)+zJ_{\nu }^{\prime }(z).$

\textbf{ii.} If $\nu \in (-2,-1),$ then the radius of uniform convexity of
the function $g_{\nu }$ is$\ r^{uc}(g_{\nu }),$ where $r^{uc}(g_{\nu })$ is
the unique root of the equation 
\begin{equation*}
1+2r\frac{rI_{\nu }(r)-(2\nu -1)I_{\nu +1}(r)}{I_{\nu }(r)+rI_{\nu +1}(r)}=0,
\end{equation*}%
in the interval $(0,a).$
\end{theorem}

\begin{proof}
First we prove part \emph{i}\textbf{\ }for $\nu >-1$ and later part \emph{ii}%
\textbf{\ }for\textbf{\ }$\nu \in (-2,-1).$ In \cite[Lemma 2.4]{Ba1} has
been proven the equality 
\begin{equation}
\frac{zg_{\nu }^{\prime \prime }(z)}{g_{\nu }^{\prime }(z)}=z\frac{zJ_{\nu
+2}(z)-3J_{\nu +1}(z)}{J_{\nu }(z)-zJ_{\nu +1}(z)}=-\sum_{n\geq 1}\frac{%
2z^{2}}{\alpha _{\nu ,n}^{2}-z^{2}}  \label{12}
\end{equation}%
where $\alpha _{\nu ,n}$ is the $n$th positive zeros of the Dini function $%
z\mapsto (1-\nu )J_{\nu }(z)+zJ_{\nu }^{\prime }(z).$ Using this equality,
in \cite{Ba1} the following inequality has been proven 
\begin{equation}
\func{Re}\left( 1+\frac{zg_{\nu }^{\prime \prime }(z)}{g_{\nu }^{\prime }(z)}%
\right) \geq 1+\frac{rg_{\nu }^{\prime \prime }(r)}{g_{\nu }^{\prime }(r)},\
|z|\leq {r}<\alpha _{\nu ,1}.  \label{13}
\end{equation}%
Equality (\ref{12}) also implies that%
\begin{eqnarray}
\left\vert \frac{zg_{\nu }^{\prime \prime }(z)}{g_{\nu }^{\prime }(z)}%
\right\vert &=&\left\vert \sum_{n\geq 1}\frac{2z^{2}}{\alpha _{\nu
,n}^{2}-z^{2}}\right\vert \leq \sum_{n\geq 1}\left\vert \frac{2z^{2}}{\alpha
_{\nu ,n}^{2}-z^{2}}\right\vert \leq \sum_{n\geq 1}\frac{2r^{2}}{\alpha
_{\nu ,n}^{2}-r^{2}}  \label{14} \\
&=&-\frac{rg_{\nu }^{\prime \prime }(r)}{g_{\nu }^{\prime }(r)},\ |z|\leq {r}%
<\alpha _{\nu ,1}.\ \ \   \notag
\end{eqnarray}%
Now summarizing (\ref{13}) and (\ref{14}) we get 
\begin{equation}
\func{Re}\left( 1+\frac{zg_{\nu }^{\prime \prime }(z)}{g_{\nu }^{\prime }(z)}%
\right) -\left\vert \frac{zg_{\nu }^{\prime \prime }(z)}{g_{\nu }^{\prime
}(z)}\right\vert \geq 1+2\frac{rg_{\nu }^{\prime \prime }(r)}{g_{\nu
}^{\prime }(r)},\ |z|\leq {r}<\alpha _{\nu ,1}.  \label{141}
\end{equation}%
The equality holds if and only if $z=r.$ Finally we get 
\begin{equation*}
\inf_{|z|<r}\left[ \func{Re}\left( 1+\frac{zg_{\nu }^{\prime \prime }(z)}{%
g_{\nu }^{\prime }(z)}\right) -\left\vert \frac{zg_{\nu }^{\prime \prime }(z)%
}{g_{\nu }^{\prime }(z)}\right\vert \right] =1+2\frac{rg_{\nu }^{\prime
\prime }(r)}{g_{\nu }^{\prime }(r)},\ \ r\in (0,\alpha _{\nu ,1}).
\end{equation*}%
The mapping $\varphi _{\nu }:(0,\alpha _{\nu ,1})\rightarrow \mathbb{R}$
defined by $\varphi _{\nu }(r)=1+2\frac{rg_{\nu }^{\prime \prime }(r)}{%
g_{\nu }^{\prime }(r)}=1-2\sum_{n\geq 1}\frac{2r^{2}}{\alpha _{\nu
,n}^{2}-r^{2}}$ is strictly decreasing, $\lim_{r\searrow 0}\varphi _{\nu
}(r)=1$ and $\lim_{r\nearrow {\alpha _{\nu ,1}}}\varphi _{\nu }(r)=-\infty .$
By using the recurrence relation $2\nu J_{\nu }(z)=z\left[ J_{\nu
-1}(z)+J_{\nu +1}(z)\right] $, it follows that the equation $1+2\frac{%
rg_{\nu }^{\prime \prime }(r)}{g_{\nu }^{\prime }(r)}=1+2r\frac{(2\nu
-1)J_{\nu +1}(r)-rJ_{\nu }(r)}{J_{\nu }(r)-rJ_{\nu +1}(r)}=0$ has a unique
root $r_{0}\in (0,\alpha _{\nu ,1})$ and $r_{0}=r^{uc}(g_{\nu }).$

\textbf{ii. }By using the result of Hurwitz \cite[p. 305]{Wat} on zeros of
Bessel functions of the first kind, the condition $\nu \in (-2,-1)$ implies $%
\alpha _{\nu ,1}=ia,\;a>0$ and $\alpha _{\nu ,n}>0$ for $n\in \{2,3,...\}.$
Thus, from equality (\ref{12}), we have 
\begin{eqnarray}
1+\frac{zg_{\nu }^{\prime \prime }(z)}{g_{\nu }^{\prime }(z)} &=&1+\frac{%
2z^{2}}{a^{2}+z^{2}}-2\sum_{n\geq 2}\frac{z^{2}}{\alpha _{v,n}^{2}-z^{2}}
\label{g1} \\
&=&1-\frac{3a^{2}}{2(1+\nu )}\frac{z^{2}}{a^{2}+z^{2}}-2\sum_{n\geq 2}\frac{%
a^{2}+\alpha _{\nu ,n}^{2}}{\alpha _{\nu ,n}^{2}}\frac{z^{4}}{%
(a^{2}+z^{2})(\alpha _{\nu ,n}^{2}-z^{2})}  \notag
\end{eqnarray}%
Here, we used following equality (see \cite[p. 305]{Ba3})%
\begin{equation*}
\sum_{n=1}^{\infty }\frac{1}{\alpha _{v,n}^{2}}=\frac{3}{4(\nu +1)}\text{
and so }\frac{1}{a^{2}}=-\frac{3}{4(\nu +1)}+\sum_{n=2}^{\infty }\frac{1}{%
\alpha _{v,n}^{2}}.\text{ }
\end{equation*}

In \cite[p. 305]{Ba3} the following equality has been proven%
\begin{eqnarray}
\func{Re}\left( 1+\frac{zg_{\nu }^{\prime \prime }(z)}{g_{\nu }^{\prime }(z)}%
\right) &\geq &1+\frac{3a^{2}}{2(1+\nu )}\frac{r^{2}}{a^{2}+r^{2}}
\label{g2} \\
&&-2\sum_{n\geq 2}\frac{a^{2}+\alpha _{\nu ,n}^{2}}{\alpha _{\nu ,n}^{2}}%
\frac{r^{4}}{(a^{2}+r^{2})(\alpha _{\nu ,n}^{2}-r^{2})}  \notag \\
&=&1+\frac{irg_{\nu }^{\prime \prime }(ir)}{g_{\nu }^{\prime }(ir)}>0,\
|z|\leq r<a.  \notag
\end{eqnarray}%
On the other hand, if in inequality (\ref{41} ) we replace $z$ by $z^{2},$
and $b$ by $\alpha _{\nu ,n},$ taking in account that $-\frac{3a^{2}}{%
2(1+\nu )}>0,$ we get the following inequalities%
\begin{eqnarray}
\left\vert \frac{zg_{\nu }^{\prime \prime }(z)}{g_{\nu }^{\prime }(z)}%
\right\vert &=&\left\vert -\frac{3a^{2}}{2(1+\nu )}\frac{z^{2}}{a^{2}+z^{2}}%
-2\sum_{n\geq 2}\frac{a^{2}+\alpha _{\nu ,n}^{2}}{\alpha _{\nu ,n}^{2}}\frac{%
z^{4}}{(a^{2}+z^{2})(\alpha _{\nu ,n}^{2}-z^{2})}\right\vert  \label{g3} \\
&\leq &-\frac{3a^{2}}{2(1+\nu )}\left\vert \frac{z^{2}}{a^{2}+z^{2}}%
\right\vert +2\sum_{n\geq 2}\frac{a^{2}+\alpha _{\nu ,n}^{2}}{\alpha _{\nu
,n}^{2}}\left\vert \frac{z^{4}}{(a^{2}+z^{2})(\alpha _{\nu ,n}^{2}-z^{2})}%
\right\vert  \notag \\
&\leq &-\frac{3a^{2}}{2(1+\nu )}\frac{r^{2}}{a^{2}-r^{2}}+2\sum_{n\geq 2}%
\frac{a^{2}+\alpha _{\nu ,n}^{2}}{\alpha _{\nu ,n}^{2}}\frac{r^{4}}{%
(a^{2}-r^{2})(\alpha _{\nu ,n}^{2}+r^{2})}  \notag \\
&=&-\frac{irg_{\nu }^{\prime \prime }(ir)}{g_{\nu }^{\prime }(ir)},\ |z|\leq 
{r}<a.\ \ \   \notag
\end{eqnarray}

From inequalites (\ref{g2}) and (\ref{g3}) we get 
\begin{equation*}
\inf_{|z|<r}\left[ \func{Re}\left( 1+\frac{zg_{\nu }^{\prime \prime }(z)}{%
g_{\nu }^{\prime }(z)}\right) -\left\vert \frac{zg_{\nu }^{\prime \prime }(z)%
}{g_{\nu }^{\prime }(z)}\right\vert \right] =1+2\frac{irg_{\nu }^{\prime
\prime }(ir)}{g_{\nu }^{\prime }(ir)}
\end{equation*}%
for every $r\in (0,a).$ Now, consider the function $\Theta _{\nu
}:(0,a)\rightarrow \mathbb{R}$, defined by 
\begin{equation*}
\Theta _{\nu }(r)=1+2\frac{irg_{\nu }^{\prime \prime }(ir)}{g_{\nu }^{\prime
}(ir)}=1-\frac{4r^{2}}{a^{2}-r^{2}}+4\sum_{n\geq 2}\frac{r^{2}}{\alpha
_{v,n}^{2}+r^{2}}.
\end{equation*}

Since $\Theta _{\nu }(r)=1+2\frac{irg_{\nu }^{\prime \prime }(ir)}{g_{\nu
}^{\prime }(ir)}=1-\frac{4r^{2}}{a^{2}-r^{2}}+4\sum_{n\geq 2}\frac{r^{2}}{%
\alpha _{v,n}^{2}+r^{2}}$ is strictly decreasing, $\lim_{r\searrow 0}\Theta
_{\nu }(r)=1$ and $\lim_{r\nearrow {a}}\Theta _{\nu }(r)=-\infty $ it
follows that the equation $1+2\frac{irg_{\nu }^{\prime \prime }(ir)}{g_{\nu
}^{\prime }(ir)}=1+2r\frac{rI_{\nu }(r)-(2\nu -1)I_{\nu +1}(r)}{I_{\nu
}(r)+rI_{\nu +1}(r)}=0$ has a unique root $r^{uc}(g_{\nu })\in (0,a).$
\end{proof}

\begin{equation*}
\FRAME{itbpFU}{3.0113in}{1.8066in}{0in}{\Qcb{The graph of the function $%
\protect\nu \mapsto 1+2r\frac{(2\protect\nu -1)J_{\protect\nu +1}(r)-rJ_{%
\protect\nu }(r)}{J_{\protect\nu }(r)-rJ_{\protect\nu +1}(r)}$ for $\protect%
\nu \in \{-0.5,0,0.5,1.5\}$ on $\left[ 0,0.86\right] $}}{}{Figure}{\special%
{language "Scientific Word";type "GRAPHIC";display "USEDEF";valid_file
"T";width 3.0113in;height 1.8066in;depth 0in;original-width
5.0004in;original-height 3.1142in;cropleft "0";croptop "0.9815";cropright
"0.9857";cropbottom "0";tempfilename 'OH5TDI01.wmf';tempfile-properties
"PR";}}\text{ \ \ \ \ }\FRAME{itbpFU}{2.93in}{1.8308in}{0in}{\Qcb{The graph
of the function $\protect\nu \mapsto 1+2r\frac{rI_{\protect\nu }(r)-(2%
\protect\nu -1)I_{\protect\nu +1}(r)}{I_{\protect\nu }(r)+rI_{\protect\nu %
+1}(r)}$ for $\protect\nu \in \{-1.8,-1.5,-1.4,-1.2\}$ on $\left[ 0,0.5%
\right] $}}{}{Figure}{\special{language "Scientific Word";type
"GRAPHIC";maintain-aspect-ratio TRUE;display "USEDEF";valid_file "T";width
2.93in;height 1.8308in;depth 0in;original-width 5.0004in;original-height
3.1142in;cropleft "0";croptop "1";cropright "1";cropbottom "0";tempfilename
'OH5TDJ02.wmf';tempfile-properties "PR";}}
\end{equation*}

\begin{theorem}
\label{th4} \textbf{i.} If $\nu >-1,$ then the radius of uniform convexity
of the function $h_{\nu }$ is the smallest positive root of the equation 
\begin{equation*}
1+r^{\frac{1}{2}}\frac{2(\nu -1)J_{\nu +1}(r^{\frac{1}{2}})-r^{\frac{1}{2}%
}J_{\nu }(r^{\frac{1}{2}})}{2J_{\nu }(r^{\frac{1}{2}})-r^{\frac{1}{2}}J_{\nu
+1}(r^{\frac{1}{2}})}=0.
\end{equation*}%
Moreover, $r_{\alpha }^{c}(h_{\nu })<\beta _{\nu ,1}^{2}<j_{\nu ,1}^{2},$
where $\beta _{\nu ,1}$ is the first positive zero of the Dini function $%
z\mapsto (2-\nu )J_{\nu }(z)+zJ_{\nu }^{\prime }(z).$

\textbf{ii.} If $\nu \in (-2,-1),$ then the radius of uniform convexity of
the function $h_{\nu }$ is $r^{uc}(g_{\nu }),$ where $r^{uc}(g_{\nu })$ is
the unique root of the equation 
\begin{equation*}
1+r^{\frac{1}{2}}\frac{r^{\frac{1}{2}}I_{\nu }(r^{\frac{1}{2}})-2(\nu
-1)I_{\nu +1}(r^{\frac{1}{2}})}{2I_{\nu }(r^{\frac{1}{2}})+r^{\frac{1}{2}%
}I_{\nu +1}(r^{\frac{1}{2}})}=0,
\end{equation*}%
in the interval $(0,a).$
\end{theorem}

\begin{proof}
\textbf{i. }In \cite[Lemma 2.5]{Ba1} the following equality has been proven 
\begin{equation}
\frac{zh_{\nu }^{\prime \prime }(z)}{h_{\nu }^{\prime }(z)}=\frac{\nu (\nu
-2)J_{\nu }(z^{\frac{1}{2}})+(3-2\nu )z^{\frac{1}{2}}J_{\nu }^{\prime }(z^{%
\frac{1}{2}})+zJ_{\nu }^{\prime \prime }(z^{\frac{1}{2}})}{2(2-\nu )J_{\nu
}(z^{\frac{1}{2}})+2z^{\frac{1}{2}}J_{\nu }^{\prime }(z^{\frac{1}{2}})}%
=-\sum_{n\geq 1}\frac{z}{\beta _{\nu ,n}^{2}-z},  \label{15}
\end{equation}%
and in the same paper, using the above equality, in Theorem 1.3 the
following inequality was deduced 
\begin{equation}
\func{Re}\left( 1+\frac{zh_{\nu }^{\prime \prime }(z)}{h_{\nu }^{\prime }(z)}%
\right) \geq 1+\frac{rh_{\nu }^{\prime \prime }(r)}{h_{\nu }^{\prime }(r)}=1+%
\frac{r^{\frac{1}{2}}}{2}\frac{r^{\frac{1}{2}}J_{\nu +2}(r^{\frac{1}{2}%
})-4J_{\nu +1}(r^{\frac{1}{2}})}{2J_{\nu }(r^{\frac{1}{2}})-r^{\frac{1}{2}%
}J_{\nu +1}(r^{\frac{1}{2}})},\ \ \ \ \   \label{16}
\end{equation}%
where $|z|<r<\beta _{\nu ,1}^{2}<j_{\nu ,1}^{2}$ and $\beta _{\nu ,n}$ is
the $n$ th positive zeros of the Dini function $z\mapsto (2-\nu )J_{\nu
}(z)+zJ_{\nu }^{\prime }(z).$ The equality (\ref{15}) and the second
inequality of (\ref{4}) imply%
\begin{eqnarray}
\left\vert \frac{zh_{\nu }^{\prime \prime }(z)}{h_{\nu }^{\prime }(z)}%
\right\vert &=&\left\vert \sum_{n\geq 1}\frac{z}{\beta _{\nu ,n}^{2}-z}%
\right\vert \leq \sum_{n\geq 1}\left\vert \frac{z}{\beta _{\nu ,n}^{2}-z}%
\right\vert \leq \sum_{n\geq 1}\frac{r}{\beta _{\nu ,n}^{2}-r}  \label{17} \\
&=&-\frac{rh_{\nu }^{\prime \prime }(r)}{h_{\nu }^{\prime }(r)},\ \
|z|<r<\beta _{\nu ,1}^{2}.  \notag
\end{eqnarray}%
From the inequalities (\ref{16}) and (\ref{17}) we infer 
\begin{equation}
\func{Re}\left( 1+\frac{zh_{\nu }^{\prime \prime }(z)}{h_{\nu }^{\prime }(z)}%
\right) -\left\vert \frac{zh_{\nu }^{\prime \prime }(z)}{h_{\nu }^{\prime
}(z)}\right\vert \geq 1+2r\frac{h_{\nu }^{\prime \prime }(r)}{h_{\nu
}^{\prime }(r)}\ \ |z|<r<\beta _{\nu ,1}^{2}.  \label{171}
\end{equation}%
The equality holds if and only if $z=r.$ Thus we get 
\begin{eqnarray}
\inf_{|z|<r}\left[ \func{Re}\left( 1+\frac{zh_{\nu }^{\prime \prime }(z)}{%
h_{\nu }^{\prime }(z)}\right) -\left\vert \frac{zh_{\nu }^{\prime \prime }(z)%
}{h_{\nu }^{\prime }(z)}\right\vert \right] &=&1+2r\frac{h_{\nu }^{\prime
\prime }(r)}{h_{\nu }^{\prime }(r)}  \notag \\
&=&1+r^{\frac{1}{2}}\frac{r^{\frac{1}{2}}J_{\nu +2}(r^{\frac{1}{2}})-4J_{\nu
+1}(r^{\frac{1}{2}})}{2J_{\nu }(r^{\frac{1}{2}})-r^{\frac{1}{2}}J_{\nu
+1}(r^{\frac{1}{2}})},  \notag
\end{eqnarray}%
for every $r\in (0,\beta _{\nu ,1}^{2}).$ Since the mapping $\phi _{\nu
}:(0,\beta _{\nu ,1}^{2})\rightarrow \mathbb{R}$ defined by $\phi _{\nu
}(r)=1+2r\frac{h_{\nu }^{\prime \prime }(r)}{h_{\nu }^{\prime }(r)}%
=1-\sum_{n\geq 1}\frac{2r}{\beta _{\nu ,n}^{2}-r}$ is strictly decreasing,
and $\lim_{r\searrow 0}\phi _{\nu }(r)=1$ and $\lim_{r\nearrow {\alpha _{\nu
,1}}}\phi _{\nu }(r)=-\infty ,$ it follows that the equation $1+r^{\frac{1}{2%
}}\frac{r^{\frac{1}{2}}J_{\nu +2}(r^{\frac{1}{2}})-4J_{\nu +1}(r^{\frac{1}{2}%
})}{2J_{\nu }(r^{\frac{1}{2}})-r^{\frac{1}{2}}J_{\nu +1}(r^{\frac{1}{2}})}%
=1+r^{\frac{1}{2}}\frac{2(\nu -1)J_{\nu +1}(r^{\frac{1}{2}})-r^{\frac{1}{2}%
}J_{\nu }(r^{\frac{1}{2}})}{2J_{\nu }(r^{\frac{1}{2}})-r^{\frac{1}{2}}J_{\nu
+1}(r^{\frac{1}{2}})}=0$ has a unique root $r_{0}\in (0,\beta _{\nu ,1}^{2})$%
, and this root is the radius of uniform convexity. In the last equality we
use the recurrence relation $2\nu J_{\nu }(z)=z\left[ J_{\nu -1}(z)+J_{\nu
+1}(z)\right] $.

\textbf{ii. }By using the result of Hurwitz \cite[p. 305]{Wat} on zeros of
Bessel functions of the first kind, the condition $\nu \in (-2,-1)$ implies $%
\beta _{\nu ,1}=ib,\;b>0$ and $0<\beta _{\nu ,2}<\beta _{\nu ,3}<...\beta
_{\nu ,n}<...$ for $n\in \{2,3,...\}.$ Thus, from equality (\ref{15}), we
have%
\begin{eqnarray*}
1+\frac{zh_{\nu }^{\prime \prime }(z)}{h_{\nu }^{\prime }(z)} &=&1+\frac{z}{%
b^{2}+z}-\sum_{n\geq 2}\frac{z}{\beta _{\nu ,n}^{2}-z} \\
&=&1-\frac{b^{2}}{2(1+\nu )}\frac{z}{b^{2}+z}-\sum_{n\geq 2}\frac{%
b^{2}+\beta _{\nu ,n}^{2}}{\beta _{\nu ,n}^{2}}\frac{z^{2}}{(b^{2}+z)(\beta
_{\nu ,n}^{2}-z)}.
\end{eqnarray*}%
Here, we used the following equality (see \cite[p. 305]{Ba3})%
\begin{equation*}
\sum_{n=1}^{\infty }\frac{1}{\beta _{\nu ,n}^{2}}=\frac{1}{2(\nu +1)}\text{
and so }\frac{1}{b^{2}}=\frac{1}{2(\nu +1)}-\sum_{n=2}^{\infty }\frac{1}{%
\beta _{\nu ,n}^{2}}.\text{ }
\end{equation*}

In \cite[p. 305]{Ba3} the following equality has been proven%
\begin{eqnarray*}
\func{Re}\left( 1+\frac{zh_{\nu }^{\prime \prime }(z)}{h_{\nu }^{\prime }(z)}%
\right) &\geq &1+\frac{b^{2}}{2(1+\nu )}\frac{r}{b^{2}-r}-\sum_{n\geq 2}%
\frac{b^{2}+\beta _{\nu ,n}^{2}}{\beta _{\nu ,n}^{2}}\frac{r^{2}}{%
(b^{2}-r)(\beta _{\nu ,n}^{2}+r)} \\
&=&1+\frac{-rh_{\nu }^{\prime \prime }(-r)}{h_{\nu }^{\prime }(-r)}>0,\
|z|\leq r<b^{2}.
\end{eqnarray*}%
On the other hand, from inequality (\ref{41} ) we get%
\begin{eqnarray*}
\left\vert \frac{zh_{\nu }^{\prime \prime }(z)}{h_{\nu }^{\prime }(z)}%
\right\vert &=&\left\vert -\frac{b^{2}}{2(1+\nu )}\frac{z}{b^{2}+z}%
-\sum_{n\geq 2}\frac{b^{2}+\beta _{\nu ,n}^{2}}{\beta _{\nu ,n}^{2}}\frac{%
z^{2}}{(b^{2}+z)(\beta _{\nu ,n}^{2}-z)}\right\vert \\
&\leq &-\frac{b^{2}}{2(1+\nu )}\left\vert \frac{z}{b^{2}+z}\right\vert
+\sum_{n\geq 2}\frac{b^{2}+\beta _{\nu ,n}^{2}}{\beta _{\nu ,n}^{2}}%
\left\vert \frac{z^{2}}{(b^{2}+z)(\beta _{\nu ,n}^{2}-z)}\right\vert \\
&\leq &-\frac{b^{2}}{2(1+\nu )}\frac{r}{b^{2}-r}+\sum_{n\geq 2}\frac{%
b^{2}+\beta _{\nu ,n}^{2}}{\beta _{\nu ,n}^{2}}\frac{r^{2}}{(b^{2}-r)(\beta
_{\nu ,n}^{2}+r)} \\
&=&-\frac{-rh_{\nu }^{\prime \prime }(-r)}{h_{\nu }^{\prime }(-r)},\ |z|\leq 
{r}<b^{2}.\ \ \ 
\end{eqnarray*}

Consequently the following inequality holds 
\begin{equation*}
\inf_{|z|<r}\left[ \func{Re}\left( 1+\frac{zh_{\nu }^{\prime \prime }(z)}{%
h_{\nu }^{\prime }(z)}\right) -\left\vert \frac{zh_{\nu }^{\prime \prime }(z)%
}{h_{\nu }^{\prime }(z)}\right\vert \right] =1+2\frac{-rh_{\nu }^{\prime
\prime }(-r)}{h_{\nu }^{\prime }(-r)}
\end{equation*}%
for every $r\in (0,b^{2}).$ Since the mapping 
\begin{equation*}
\Phi _{\nu }:(0,b^{2})\rightarrow \mathbb{R},\ \ \Phi _{\nu }(r)=1+2\frac{%
-rh_{\nu }^{\prime \prime }(-r)}{h_{\nu }^{\prime }(-r)}=1-2\frac{r}{b^{2}-r}%
+2\sum_{n\geq 2}\frac{r}{\beta _{\nu ,n}^{2}+r}
\end{equation*}%
is strictly decreasing, $\lim_{r\searrow 0}\Phi _{\nu }=1$ and $%
\lim_{r\nearrow {b}^{2}}\Phi _{\nu }=-\infty .$ Thus it follows that the
equation $1+2\frac{-rh_{\nu }^{\prime \prime }(-r)}{h_{\nu }^{\prime }(-r)}%
=1+r^{\frac{1}{2}}\frac{r^{\frac{1}{2}}I_{\nu }(r^{\frac{1}{2}})-2(\nu
-1)I_{\nu +1}(r^{\frac{1}{2}})}{2I_{\nu }(r^{\frac{1}{2}})+r^{\frac{1}{2}%
}I_{\nu +1}(r^{\frac{1}{2}})}=0$ has a unique root $r^{uc}(h_{\nu })\in
(0,b^{2}).$
\end{proof}

\begin{equation*}
\FRAME{itbpFU}{2.9196in}{1.8248in}{0in}{\Qcb{The graph of the function $%
\protect\nu \mapsto 1+r^{\frac{1}{2}}\frac{2(\protect\nu -1)J_{\protect\nu %
+1}(r^{\frac{1}{2}})-r^{\frac{1}{2}}J_{\protect\nu }(r^{\frac{1}{2}})}{2J_{%
\protect\nu }(r^{\frac{1}{2}})-r^{\frac{1}{2}}J_{\protect\nu +1}(r^{\frac{1}{%
2}})}$ for $\protect\nu \in \{-0.5,0,0.5,1.5\}$ on $\left[ 0,1\right] $}}{}{%
Figure}{\special{language "Scientific Word";type
"GRAPHIC";maintain-aspect-ratio TRUE;display "USEDEF";valid_file "T";width
2.9196in;height 1.8248in;depth 0in;original-width 5.0004in;original-height
3.1142in;cropleft "0";croptop "1";cropright "1";cropbottom "0";tempfilename
'OH5TDJ03.wmf';tempfile-properties "PR";}}\text{ \ \FRAME{itbpFU}{2.9404in}{%
1.8377in}{0in}{\Qcb{The graph of the function $\protect\nu \mapsto 1+r^{%
\frac{1}{2}}\frac{r^{\frac{1}{2}}I_{\protect\nu }(r^{\frac{1}{2}})-2(\protect%
\nu -1)I_{\protect\nu +1}(r^{\frac{1}{2}})}{2I_{\protect\nu }(r^{\frac{1}{2}%
})+r^{\frac{1}{2}}I_{\protect\nu +1}(r^{\frac{1}{2}})}$ for $\protect\nu \in
\{-1.8,-1.5,-1.4,-1.2\}$ on $\left[ 0,0.35\right] $}}{}{Figure}{\special%
{language "Scientific Word";type "GRAPHIC";maintain-aspect-ratio
TRUE;display "USEDEF";valid_file "T";width 2.9404in;height 1.8377in;depth
0in;original-width 5.0004in;original-height 3.1142in;cropleft "0";croptop
"1";cropright "1";cropbottom "0";tempfilename
'OH5TDJ04.wmf';tempfile-properties "PR";}}\ }
\end{equation*}

\subsection{\textbf{Uniform convexity of normalized Bessel functions}}

Now, let us recall some results on the geometric behavior of the functions $%
f_{\nu }(z),$ $g_{\nu }(z)$ and $h_{\nu }(z).$ In 2010 Szász \cite{Sza0}
investigated the starlikeness of $h_{\nu }(z)$ in case of $\nu \geq \nu
_{\ast }\simeq -0.5623...,$ where $\nu _{\ast }$ is the unique root of the
equation $f_{\nu }^{\prime }(1)=0$ in $(-1,1).$ Baricz and Szász \cite{Ba1}
proved that $f_{\nu }(z)$ and $g_{\nu }(z)$ are convex in $U$ $%
\Leftrightarrow $ $\nu \geq 1,$ and $h_{\nu }(z)$ is convex in $U$ $%
\Leftrightarrow $ $\nu \geq \nu _{\ast \ast }\simeq -0.1438...,$ where $\nu
_{\ast \ast }$ is the unique root of the equation $(2\nu -4)J_{\nu
+1}(1)+3J_{\nu }(1)=0.$ Furthermore, Baricz and Szász \cite{Ba4}, Baricz et
al. \cite{Ba5} and Baricz et al. \cite{Ba6} obtained necessary and
sufficient conditions for the starlikeness and close-to-convexity of the
function $h_{\nu }(z)$ and its derivatives, some special combinations of
Bessel functions and their derivatives, and the functions $f_{\nu }(z),$ $%
g_{\nu }(z)$ and derivatives of $h_{\nu }(z)$ in $U,$ respectively, by using
a result of Shah and Trimble (see \cite[Theorem 2]{Sh}) about transcendental
entire functions with univalent derivatives. In this section, we deal with
the uniform convexity of the normalized Bessel functions $f_{\nu }(z),$ $%
g_{\nu }(z)$ and $h_{\nu }(z)$ in $U.$

\begin{theorem}
\label{th5} The function $f_{\nu }$ is uniformly convex in $U$ if and only
if $\nu >\nu _{1}\simeq 1.4426...,$ where $\nu _{1}$ is the unique root of
the equation 
\begin{equation*}
\nu (3\nu -2)\left( J_{\nu }(1)\right) ^{2}+\nu (4\nu -5)J_{\nu }(1)J_{\nu
-1}(1)+2(1-\nu )\left( J_{\nu -1}(1)\right) ^{2}=0
\end{equation*}%
situated in $(\nu ^{\ast },\infty ),$ where $\nu ^{\ast }\simeq 0.39001...$
is the root of the equation $J_{\nu }^{\prime }(1)=0.$
\end{theorem}

\begin{proof}
According to (\ref{111}) for $z\in U$ we know that%
\begin{eqnarray*}
\func{Re}\left( 1+\frac{zf_{\nu }^{\prime \prime }(z)}{f_{\nu }^{\prime }(z)}%
\right) -\left\vert \frac{zf_{\nu }^{\prime \prime }(z)}{f_{\nu }^{\prime
}(z)}\right\vert &\geq &1-2\sum_{n\geq 1}\left( \frac{2r^{2}}{j_{\nu
,n}^{\prime 2}-r^{2}}-\left( 1-\frac{1}{\nu }\right) \frac{2r^{2}}{j_{\nu
,n}^{2}-r^{2}}\right) \\
&=&1+2\frac{rJ_{\nu }^{\prime \prime }(r)}{J_{\nu }^{\prime }(r)}+2\left( 
\frac{1}{\nu }-1\right) \frac{rJ_{\nu }^{\prime }(r)}{J_{\nu }(r)}=1+2\frac{%
rf_{\nu }^{\prime \prime }(r)}{f_{\nu }^{\prime }(r)}.
\end{eqnarray*}%
Since the mapping $\psi _{\nu }:(0,j_{\nu ,1}^{\prime })\rightarrow \mathbb{R%
}$ defined by $\psi _{\nu }(r)=1+2\frac{rf_{\nu }^{\prime \prime }(r)}{%
f_{\nu }^{\prime }(r)}$ is strictly decreasing and the inequalities $%
1<j_{\nu ,n}^{\prime }<j_{\nu ,n}$ hold for $n\in \{1,2,...\}$, it follows
that 
\begin{eqnarray*}
\func{Re}\left( 1+\frac{zf_{\nu }^{\prime \prime }(z)}{f_{\nu }^{\prime }(z)}%
\right) -\left\vert \frac{zf_{\nu }^{\prime \prime }(z)}{f_{\nu }^{\prime
}(z)}\right\vert &\geq &1-2\sum_{n\geq 1}\left( \frac{2r^{2}}{j_{\nu
,n}^{\prime 2}-r^{2}}-\left( 1-\frac{1}{\nu }\right) \frac{2r^{2}}{j_{\nu
,n}^{2}-r^{2}}\right) \\
&\geq &1-2\sum_{n\geq 1}\left( \frac{2}{j_{\nu ,n}^{\prime 2}-1}-\left( 1-%
\frac{1}{\nu }\right) \frac{2}{j_{\nu ,n}^{2}-1}\right) \\
&=&1+2\frac{J_{\nu }^{\prime \prime }(1)}{J_{\nu }^{\prime }(1)}+2\left( 
\frac{1}{\nu }-1\right) \frac{J_{\nu }^{\prime }(1)}{J_{\nu }(1)}=1+2\frac{%
f_{\nu }^{\prime \prime }(1)}{f_{\nu }^{\prime }(1)}.
\end{eqnarray*}%
Now, consider the function $\psi :(\nu ^{\ast },\infty )\rightarrow 
\mathbb{R}
,$ defined by 
\begin{equation*}
\psi (\nu )=1+2\frac{f_{\nu }^{\prime \prime }(1)}{f_{\nu }^{\prime }(1)}=1+2%
\frac{J_{\nu }^{\prime \prime }(1)}{J_{\nu }^{\prime }(1)}+2\left( \frac{1}{%
\nu }-1\right) \frac{J_{\nu }^{\prime }(1)}{J_{\nu }(1)}.
\end{equation*}%
We know that $\nu \mapsto 1+$ $\frac{f_{\nu }^{\prime \prime }(1)}{f_{\nu
}^{\prime }(1)}=1+\frac{J_{\nu }^{\prime \prime }(1)}{J_{\nu }^{\prime }(1)}%
+\left( \frac{1}{\nu }-1\right) \frac{J_{\nu }^{\prime }(1)}{J_{\nu }(1)}$
is stricly increasing on $(\nu ^{\ast },\infty )$ (see \cite{Ba1}), thus the
function $\psi $ is stricly increasing on $(\nu ^{\ast },\infty )$ too$.$
Consequently, if $\nu >\nu _{1},$ then we get the inequality $\psi (\nu
)\geq \psi (\nu _{1}).$ This in turn implies that $\nu _{1}$ is the smallest
value having the property that the condition $\nu \geq \nu _{1}$ implies
that for all $z\in U$ we have 
\begin{equation*}
\func{Re}\left( 1+\frac{zf_{\nu }^{\prime \prime }(z)}{f_{\nu }^{\prime }(z)}%
\right) -\left\vert \frac{zf_{\nu }^{\prime \prime }(z)}{f_{\nu }^{\prime
}(z)}\right\vert >\psi (\nu _{1})=0.
\end{equation*}%
Thus, we proved that the function $f_{\nu }$ is uniformly convex in $U$ if
and only if $\nu >\nu _{1},$ where $\nu _{1}$ is the unique root of the
equation 
\begin{equation*}
1+2\frac{J_{\nu }^{\prime \prime }(1)}{J_{\nu }^{\prime }(1)}+2\left( \frac{1%
}{\nu }-1\right) \frac{J_{\nu }^{\prime }(1)}{J_{\nu }(1)}=0.
\end{equation*}%
The function $J_{\nu }(z)$ satisfies the Bessel differential equation $%
z^{2}w^{\prime \prime }(z)+zw^{\prime }(z)+\left( 1-v^{2}\right) w(z)=0$ and
by using the recurrence relation $zJ_{\nu }^{\prime }(z)=zJ_{\nu -1}(z)-\nu
J_{\nu }(z)$ the above equation can be rewritten as follows 
\begin{equation*}
\nu (3\nu -2)\left( J_{\nu }(1)\right) ^{2}+\nu (4\nu -5)J_{\nu }(1)J_{\nu
-1}(1)+2(1-\nu )\left( J_{\nu -1}(1)\right) ^{2}=0.
\end{equation*}%
In last equality, we used $J_{\nu }(1)>0$ and $J_{\nu }^{\prime }(1)>0$ when 
$\nu >\nu ^{\ast }$ (see \cite{Ba1}).
\end{proof}

\begin{theorem}
\label{th6} The function $g_{\nu }$ is uniformly convex in $U$ if and only
if $\nu >\nu _{2}\simeq 2.44314...,$ where $\nu _{2}$ is the unique root of
the equation 
\begin{equation*}
(4\nu -3)J_{\nu +1}(1)-J_{\nu }(1)=0
\end{equation*}%
situated in $[0,\infty ).$
\end{theorem}

\begin{proof}
Taking into account the the inequality (\ref{141}) for $z\in U$ we know that%
\begin{eqnarray*}
\func{Re}\left( 1+\frac{zg_{\nu }^{\prime \prime }(z)}{g_{\nu }^{\prime }(z)}%
\right) -\left\vert \frac{zg_{\nu }^{\prime \prime }(z)}{g_{\nu }^{\prime
}(z)}\right\vert &\geq &1-4\sum_{n\geq 1}\frac{r^{2}}{\alpha _{\nu
,n}^{2}-r^{2}} \\
&=&1+2\frac{rg_{\nu }^{\prime \prime }(r)}{g_{\nu }^{\prime }(r)},\ |z|\leq {%
r}<\alpha _{\nu ,1}.
\end{eqnarray*}%
Since the mapping $\varphi _{\nu }:[0,\alpha _{\nu ,1})\rightarrow \mathbb{R}
$ defined by $\varphi _{\nu }(r)=1+2\frac{rg_{\nu }^{\prime \prime }(r)}{%
g_{\nu }^{\prime }(r)}$ is strictly decreasing and the inequalities $%
1<\alpha _{\nu ,1}$ hold for $\nu \geq 0$ (see \cite[Lemma 2.4]{Ba1}), we
get that%
\begin{eqnarray*}
\func{Re}\left( 1+\frac{zg_{\nu }^{\prime \prime }(z)}{g_{\nu }^{\prime }(z)}%
\right) -\left\vert \frac{zg_{\nu }^{\prime \prime }(z)}{g_{\nu }^{\prime
}(z)}\right\vert &\geq &1-4\sum_{n\geq 1}\frac{r^{2}}{\alpha _{\nu
,n}^{2}-r^{2}} \\
&\geq &1-4\sum_{n\geq 1}\frac{1}{\alpha _{\nu ,n}^{2}-1}=1+2\frac{g_{\nu
}^{\prime \prime }(1)}{g_{\nu }^{\prime }(1)} \\
&=&1+2\frac{(2\nu -1)J_{\nu +1}(1)-J_{\nu }(1)}{J_{\nu }(1)-J_{\nu +1}(1)}.
\end{eqnarray*}%
Now, consider the function $\varphi :[0,\infty )\rightarrow 
\mathbb{R}
,$ defined by 
\begin{equation*}
\varphi (\nu )=1+2\frac{(2\nu -1)J_{\nu +1}(1)-J_{\nu }(1)}{J_{\nu
}(1)-J_{\nu +1}(1)}.
\end{equation*}%
We know that $\nu \mapsto 1+$ $\frac{g_{\nu }^{\prime \prime }(1)}{g_{\nu
}^{\prime }(1)}=1+\frac{(2\nu -1)J_{\nu +1}(1)-J_{\nu }(1)}{J_{\nu
}(1)-J_{\nu +1}(1)}$ is stricly increasing and $J_{\nu }(1)-J_{\nu +1}(1)>0$
on $[0,\infty )$ (see \cite{Ba1}), thus the function $\varphi $ is stricly
increasing on $[0,\infty )$ too$.$ Since $\varphi $ is stricly increasing,
it follows that if $\nu >\nu _{2}\simeq 2.44314...,$ then we get the
inequality%
\begin{equation*}
1+2\frac{(2\nu -1)J_{\nu +1}(1)-J_{\nu }(1)}{J_{\nu }(1)-J_{\nu +1}(1)}%
=\varphi (\nu )\geq \varphi (\nu _{2})=1+2\frac{(2\nu _{2}-1)J_{\nu
_{2}+1}(1)-J_{\nu _{2}}(1)}{J_{\nu _{2}}(1)-J_{\nu _{2}+1}(1)}=0.
\end{equation*}

Thus, under the conditon $\nu >\nu _{2}\simeq 2.44314...,$ we have for $z\in
U$ 
\begin{equation*}
\func{Re}\left( 1+\frac{zg_{\nu }^{\prime \prime }(z)}{g_{\nu }^{\prime }(z)}%
\right) -\left\vert \frac{zg_{\nu }^{\prime \prime }(z)}{g_{\nu }^{\prime
}(z)}\right\vert >\varphi (\nu _{2})=0.
\end{equation*}%
Consequently, we proved that the function $g_{\nu }$ is uniformly convex in $%
U$ if and only if $\nu >\nu _{2}\simeq 2.44314...,$ where $\nu _{2}$ is the
unique root of the equation 
\begin{equation*}
1+2\frac{(2\nu -1)J_{\nu +1}(1)-J_{\nu }(1)}{J_{\nu }(1)-J_{\nu +1}(1)}=0.
\end{equation*}%
In \cite{Ba1}, the authors proved that $J_{\nu }(1)-J_{\nu +1}(1)>0$ when $%
\nu \geq 0$. Thus the proof is completed.
\end{proof}

\begin{theorem}
\label{th7} The function $h_{\nu }$ is uniformly convex in $U$ if and only
if $\nu >\nu _{3}\simeq 0.30608...,$ where $\nu _{3}$ is the unique root of
the equation 
\begin{equation*}
(2\nu -3)J_{\nu +1}(1)+J_{\nu }(1)=0
\end{equation*}%
situated in $[0,\infty ).$
\end{theorem}

\begin{proof}
Taking into account the the inequality (\ref{171}) for $z\in U$ we know that%
\begin{eqnarray*}
\func{Re}\left( 1+\frac{zh_{\nu }^{\prime \prime }(z)}{h_{\nu }^{\prime }(z)}%
\right) -\left\vert \frac{zh_{\nu }^{\prime \prime }(z)}{h_{\nu }^{\prime
}(z)}\right\vert &\geq &1-\sum_{n\geq 1}\frac{2r}{\beta _{\nu ,n}^{2}-r} \\
&=&1+2\frac{rh_{\nu }^{\prime \prime }(r)}{h_{\nu }^{\prime }(r)},\
|z|<r<\beta _{\nu ,1}^{2}<j_{\nu ,1}^{2}.
\end{eqnarray*}%
Since the mapping $\phi _{\nu }:[0,\alpha _{\nu ,1})\rightarrow \mathbb{R}$
defined by $\phi _{\nu }(r)=1+2\frac{rh_{\nu }^{\prime \prime }(r)}{h_{\nu
}^{\prime }(r)}$ is strictly decreasing and the inequalities $1<\beta _{\nu
,1}$ hold for $\nu \geq 0$ (see \cite[Lemma 2.4]{Ba1}), we get that%
\begin{eqnarray*}
\func{Re}\left( 1+\frac{zh_{\nu }^{\prime \prime }(z)}{h_{\nu }^{\prime }(z)}%
\right) -\left\vert \frac{zh_{\nu }^{\prime \prime }(z)}{h_{\nu }^{\prime
}(z)}\right\vert &\geq &1-\sum_{n\geq 1}\frac{2r}{\beta _{\nu ,n}^{2}-r} \\
&\geq &1-\sum_{n\geq 1}\frac{2}{\beta _{\nu ,n}^{2}-1}=1+2\frac{h_{\nu
}^{\prime \prime }(1)}{h_{\nu }^{\prime }(1)} \\
&=&1+\frac{2(\nu -1)J_{\nu +1}(1)-J_{\nu }(1)}{2J_{\nu }(1)-J_{\nu +1}(1)}.
\end{eqnarray*}%
Now, consider the function $\phi :[0,\infty )\rightarrow 
\mathbb{R}
,$ defined by 
\begin{equation*}
\phi (\nu )=1+\frac{2(\nu -1)J_{\nu +1}(1)-J_{\nu }(1)}{2J_{\nu }(1)-J_{\nu
+1}(1)}.
\end{equation*}%
We know that $v\mapsto 1+$ $\frac{h_{\nu }^{\prime \prime }(1)}{h_{\nu
}^{\prime }(1)}=1+\frac{1}{2}\frac{2(\nu -1)J_{\nu +1}(1)-J_{\nu }(1)}{%
2J_{\nu }(1)-J_{\nu +1}(1)}$ is stricly increasing and $2J_{\nu }(1)-J_{\nu
+1}(1)>0$ on $[0,\infty )$ (see \cite{Ba1}), thus we can easily see that the
function $\phi $ is stricly increasing on $[0,\infty )$ too$.$ Since $\phi $
is stricly increasing, it follows that if $\nu >\nu _{3}\simeq 0.30608...,$
then we get the inequality%
\begin{equation*}
1+\frac{2(\nu -1)J_{\nu +1}(1)-J_{\nu }(1)}{2J_{\nu }(1)-J_{\nu +1}(1)}=\phi
(\nu )\geq \phi (\nu _{3})=1+\frac{2(\nu _{3}-1)J_{\nu _{3}+1}(1)-J_{\nu
_{3}}(1)}{2J_{\nu _{3}}(1)-J_{\nu _{3}+1}(1)}=0.
\end{equation*}

Thus, under the conditon $\nu >\nu _{3}\simeq 0.30608...,$ we have for $z\in
U$ 
\begin{equation*}
\func{Re}\left( 1+\frac{zh_{\nu }^{\prime \prime }(z)}{h_{\nu }^{\prime }(z)}%
\right) -\left\vert \frac{zh_{\nu }^{\prime \prime }(z)}{h_{\nu }^{\prime
}(z)}\right\vert >\phi (\nu _{3})=0.
\end{equation*}%
Thus, we proved that the function $h_{\nu }$ is uniformly convex in $U$ if
and only if $\nu >\nu _{3}\simeq 0.30608...,$ where $\nu _{3}$ is the unique
root of the equation 
\begin{equation*}
1+\frac{2(\nu -1)J_{\nu +1}(1)-J_{\nu }(1)}{2J_{\nu }(1)-J_{\nu +1}(1)}=0.
\end{equation*}%
In \cite{Ba1}, authors proved that $2J_{\nu }(1)-J_{\nu +1}(1)>0$ when $\nu
\geq 0$. Thus from last equality we obtain 
\begin{equation*}
(2\nu -3)J_{\nu +1}(1)+J_{\nu }(1)=0.
\end{equation*}
\end{proof}

\end{document}